\documentclass{amsart}

\usepackage[all]{xy}
\usepackage{amsmath}
\usepackage{amssymb}
\usepackage{array}
\usepackage{mathrsfs}
\usepackage{color}
\usepackage{verbatim}

\newcommand{\adjunction}[4]{\xymatrix{ #1 \ar@<1ex>[rr]^-{#3} && #2 \ar@<1ex>[ll]^-{#4}}}

\newcommand{\Q}{\mathbb{Q}}
\newcommand{\C}{\mathbb{C}}

\newcommand{\Aut}{\operatorname{Aut}}

\newcommand{\map}{\operatorname{map}}

\newtheorem{theorem}{Theorem}[section]

\newtheorem{proposition}[theorem]{Proposition}
\newtheorem{corollary}[theorem]{Corollary}
\newtheorem{lemma}[theorem]{Lemma}

\theoremstyle{definition}
\newtheorem{definition}[theorem]{Definition}
\newtheorem{remark}[theorem]{Remark}
\newtheorem{example}[theorem]{Example}

\def\arbreABC{\vcenter{\xymatrix@R=1pt@C=1pt{
&&&&&&\\
&*{}\ar@{-}[ur] &*{\cdot}&*{\cdot}&*{\cdot}&& \\
&&&&&&\\
&&&*{}\ar@{-}[uuurrr] \ar@{-}[uuulll] \ar@{-}[d] &&&\\
&&&&&&
}}}

\def\arbreBAC{\vcenter{\xymatrix@R=1pt@C=1pt{
&&&&&&\\
&&& &&& \\
&& &&&&\\
&&&*{}\ar@{-}[uuurrr] \ar@{-}[uuulll] \ar@{-}[d] &&&\\
&&&&&&
}}}

\def\arbreACA{\vcenter{\xymatrix@R=1pt@C=1pt{
&&&&&&\\
&*{}\ar@{-}[ur] &&&&*{}\ar@{-}[ul] & \\
&&&&&&\\
&&&*{}\ar@{-}[uuurrr] \ar@{-}[uuulll] \ar@{-}[d] &&&\\
&&&&&&
}}}

\def\arbreCAB{\vcenter{\xymatrix@R=1pt@C=1pt{
&&&&&&\\
&&&*{}\ar@{-}[ur] &&& \\
&&&&*{}\ar@{-}[uull] &&\\
&&&*{}\ar@{-}[uuurrr] \ar@{-}[uuulll] \ar@{-}[d] &&&\\
&&&&&&
}}}

\def\arbreCBA{\vcenter{\xymatrix@R=1pt@C=1pt{
&&&&&&\\
&&&*{\cdots}&&& \\
&&&&&&\\
&&&*{}\ar@{-}[uuulll] \ar@{-}[uuurr]\ar@{-}[uuull]\ar@{-}[uuurrr] \ar@{-}[d] &&&\\
&&&&&&
}}}

\def\arbreABCD{\vcenter{\xymatrix@R=1pt@C=1pt{
&&&&&&&&\\
&*{}\ar@{-}[ur] &&&&&&& \\
&&*{}\ar@{-}[uurr] &&&&&&\\
&&&*{}\ar@{-}[uuurrr] &&&&&\\
&&&&*{}\ar@{-}[uuuurrrr] \ar@{-}[uuuullll] \ar@{-}[d] &&&&\\
&&&&&&
}}}

\def\arbreACAD{\vcenter{\xymatrix@R=1pt@C=1pt{
&&&&&&&&\\
&*{}\ar@{-}[ur] &&&&*{}\ar@{-}[ul] &&& \\
&&&&&&&&\\
&&&*{}\ar@{-}[uuurrr] &&&&&\\
&&&&*{}\ar@{-}[uuuurrrr] \ar@{-}[uuuullll] \ar@{-}[d] &&&&\\
&&&&&&
}}}

\def\arbreBADA{\vcenter{\xymatrix@R=1pt@C=1pt{
&&&&&&&&\\
&&&*{}\ar@{-}[ul] &&&&*{}\ar@{-}[ul] & \\
&&*{}\ar@{-}[uurr] &&&&&&\\
&&& &&&&&\\
&&&&*{}\ar@{-}[uuuurrrr] \ar@{-}[uuuullll] \ar@{-}[d] &&&&\\
&&&&&&
}}}

\numberwithin{equation}{section}

\begin{document}
\author{Sang Xie}
\address{School of Mathematical Sciences and LPMC,
         Nankai University,
         Tianjin 300071, P.R.China}

\email{xiesang@mail.nankai.edu.cn}

\author{Jian Liu}
\address{School of Mathematical Sciences and LPMCC,
         Nankai University,
         Tianjin 300071, P.R.China}

\email{
liujian2@mail.nankai.edu.cn}
\author{Xiugui Liu*}
\address{School of Mathematical Sciences and LPMC,
         Nankai University,
         Tianjin 300071, P.R.China}
\email{xgliu@nankai.edu.cn}
\thanks{* Corresponding author\\
\indent The authors were supported in part by the National Natural Science Foundation of China (Grant No. 11571186), Tianjin Natural Science Foundation (Grant No. 19JCYBJC30200) and
the Fundamental Research Funds for the Central Universities, Nankai University (Grant No. 63191702).}
\date{}
\title{Rational homotopy type of mapping spaces via cohomology algebras
}
\maketitle
$\textbf{Abstract}$: In this paper, we show that for finite $CW$-complexes $X$ and two-stage space $Y$ (for example $n$-spheres $S^n$, homogeneous spaces and $F_0$-spaces), the rational homotopy type of $\map(X, Y)$ is determined by the cohomology algebra $H^*(X; \Q)$ and the rational homotopy type of $Y$. From this, we deduce the existence of H-structures on a component of the mapping space $\map(X, Y)$, assuming the cohomology algebras of $X$ and $Y$ are isomorphism. Finally, we will show that $\map(X, Y; f)\simeq\map(X, Y; f')$ if the corresponding  \emph{Maurer-Cartan elements} are connected by an algebra automorphism of $H^\ast(X, \Q)$. 

\maketitle
\section{Introduction}
 
For two topological spaces $X$ and $Y$, let $\map(X, Y)$ be the mapping space of all (free, continuous) maps of  $X$ into $Y$. A fundamental problem is to describe the homotopy type of the mapping space $\map(X, Y)$ in terms of the homotopy types of $X$ and $Y$. Following the work of Haefliger \cite{82AH}, M\o{}ller-Raussen \cite{85M-R} described the the rational homotopy classification problem for the components of some mapping spaces $\map(X, Y)$. In particular, M\o{}ller-Raussen gave an elegant formula for the rational homotopy type of the components of $\map(X, S^n_\Q)$ in terms of the cohomology algebra $H^*(X; \Q)$ and the Sullivan model of $S^n_\Q$. In fact, M\o{}ller-Raussen proved that if $X$ is rationally $(n)$-coconnected space, i.e. $H^{\ge n}(X;\Q)=0$ and the rational space $Y$ has the Sullivan model of the form $(\Lambda (x, y), dy=x^m)$ with $|y|\ge n$ ($|y|$ denotes the degree of $y$), then the rational homotopy type of the components of $\map(X, Y)$ is depended on the cohomology algebra $H^*(X; \Q)$ and the Sullivan model of $Y$.

Recently, the rational homotopy type of mapping spaces seem to be well described in terms of the theory of $L_\infty$-algebra \cite{15AB,13B-M,16B-G,13AL}. By using the homotopy transfer theorem for $L_\infty$-algebra, Buijs and Guti\'errez \cite{16B-G} gave an $L_\infty$-model for mapping space $\map(X, Y)$ with the cohomology of $X$ and the $L_\infty$-model of $Y$. Inspired by their work, we extend M\o{}ller-Raussen's  result by the following theorem
\begin{theorem}\label{theorem 1}
Let $X$ be a connected and finite nilpotent CW-complex, and let $Y$ be a connected rational space that has the minimal Sullivan model of the form 
\begin{equation}\label{eq1}
(\Lambda(P\oplus Q), dP=0,  dQ\subset \Lambda P),
\end{equation}
where $P$ and $Q$ are graded spaces of finite type. Then the rational homotopy type of $\map(X, Y)$ is  determined by the cohomology algebra $H^*(X; \Q)$ and the Sullivan algebra $(\Lambda(P\oplus Q), d)$.
\end{theorem}
\begin{remark}
In Theorem 1.4 of \cite{97B-S} it is shown that it is always possible to obtain a model of $\map(X, Y)$ of the form $(\Lambda (V\otimes H), d)$, where $(\Lambda V, d^\prime)$ is a Sullivan model of $Y$ and $H=H_{\ast}(X; \Q)$. But the differential $d^\prime$ depends on the structure of the coalgebra model of $X$ instead of the coproduct of $H_\ast(X; \Q)$, i.e. the product of $H^\ast(X; \Q)$.
\end{remark}

\begin{remark}
Note that $S^n$ or $\C P^n$ satisfy the hypotheses of $Y$ in Theorem \ref{theorem 1}. 
 More generally, the pure Sullivan algebras and the Sullivan model of homogeneous spaces and $F_0$-spaces also satisfy the hypotheses of $Y$ in Theorem \ref{theorem 1} \cite{GTM205}. 
\end{remark}
 
G. Lupton and S. B. Smith showed that for a CW-complex $Y$, the space $\map(Y, Y; id)$ is a group-like space \cite{08L-S}. We show that
\begin{corollary}Let $Y$ be as in the Theorem \ref{theorem 1}. If $H^\ast(X;\Q)\simeq H^\ast(Y;\Q)$, then there exist a component of mapping space $\map(X,Y)$ which is a rational H-space.
\end{corollary}

Let $X$ and $Y$ be as in Theorem \ref{theorem 1}. Given an isomorphism $\psi\colon H^\ast(X;\Q)\rightarrow H^{\ast} (X;\Q)$ between the cohomology algebra. Let  $\mathsf{MC}(H^\ast(X; \Q)\otimes L)$ denote the \emph{Maurer-Cartan} set of $H^\ast(X; \Q)\otimes L$ (see Section 3), we have 

\begin{theorem} If $h\otimes l\in \mathsf{MC}(H^\ast(X, \Q)\otimes L)$, then $\psi(h)\otimes l\in \mathsf{MC}(H\otimes L)$ and $\map(X, Y; f)\simeq\map(X, Y; f')$, where $h\otimes l$ and $\psi(h)\otimes l$ represent the maps $f\colon X\rightarrow Y$ and $f'\colon X\rightarrow Y$ respectively.
\end{theorem}

In the rest of the paper  $X$ will always denote a nilpotent finite CW-complex and $Y$ will always denote a rational finite type CW-complex. We  write $H^\ast(X)$ for $H^\ast(X, \Q)$.

The paper is organized as follows. In Section 2 a brief exposition of the $L_\infty$-model for a mapping spaces and the homotopy transfer theorem for $L_\infty$-algebra is given. Section 3 contains the proof of Theorem \ref{theorem 1} and applications.

\section{$L_\infty$-algebra and the homotopy transfer theorem}
In this section, we recall basic facts about $L_\infty$-algebras. 

\subsection{$L_\infty$-algebra}\begin{definition}
An $L_{\infty}$-algebra $(L,\{\ell_k\})$
is a graded vector space $L$ together with  linear maps 
$$\ell_k\colon L^{\otimes k}\to L,\quad x_1\otimes\cdots\otimes x_k\mapsto [x_1,\cdots,x_k],$$ of degree $k-2$,  for $k\ge 1$, satisfying the following two conditions:
\begin{itemize}
\item[{\rm (i)}] Anti-symmetry,
$$
[\cdots,x,y,\cdots]=-(-1)^{|x||y|}[\cdots,y,x,\cdots],
$$

\item[{\rm (ii)}] The generalized Jacobi identity, 
$$
\sum_{i+j=n+1}\sum_{\sigma\in S(i, n-i)}\varepsilon_{\sigma}\varepsilon(-1)^{i(j-1)}
[[x_{\sigma(1)}\ldots x_{\sigma(i)}], x_{\sigma(i+1)},\ldots, x_{\sigma(n)}]=0,
$$
where $S(i,n-i)$ denotes the set of $(i, n-i)$ shuffles.
\end{itemize}
\end{definition}
An $L_{\infty}$-algebra $(L,\{\ell_k\})$ is called \emph{minimal} if $\ell_1=0$.
For two $L_\infty$-algebras $(L,\{\ell_k\}_{k\ge 1}$ and $(L',\{\ell'_k\}_{k\ge 1}$, an $L_{\infty}$-morphism $f\colon L\to L'$ is a family of skew-symmetric linear maps $\{f^n\colon L^{\otimes n}\to L'\}$ of degree $n-1$ ($n\ge 1$) which satisfy an infinite sequence of equations involving the brackets
$\ell_k$ and $\ell'_k$ (see for instance~\cite{03MK}).
It is called an $L_{\infty}$ \emph{quasi-isomorphism} if $f^{(1)}\colon (L,\ell_1)\to(L',\ell_1')$ is a quasi-isomorphism of complexes.

Recall that a model for a space (not necessarily connected) means that a commutative differential graded algebra whose simplicial realization has the same homotopy type as the singular simplicial approximation of its rationalization \cite{78DS}. Similarly, we say an $L_\infty$-algebra model for a space means that an $L_\infty$-algebra $L$ such that $\mathscr{C}^*(L)$ is a commutative differential graded algebra model, where $\mathscr{C}^*(L)$ is the $\emph{Chevalley-Eilenberg construction}$ for $L_\infty$-algebra \cite{15AB}.

\begin{theorem}\cite[Theorem 1.4]{15AB}\label{thm2}
If $(A, d_A)$ is a finite dimensional commutative differential differential graded algebra model for $X$ and $(L,\{\ell'_k\}_{k\ge1})$ is an $L_\infty$-model for $Y$, then the following $L_\infty$-algebra is a model of $\map(X, Y)$:
$$(A\otimes L,\{\ell_k\}_{k\ge1}),$$ 
where the brackets are defined by
$$\ell_1(x\otimes a)=d_A(x)\otimes a\pm x\otimes \ell'_1(a),$$
$$[x_1\otimes a_1,\cdots,x_r\otimes a_r]=\pm x_1\cdots x_r\otimes[a_1,\cdots, a_r], \quad r>1,$$
where $x, x_1, \cdots,x_r\in A$ and $a,a_1,\cdots,a_r\in L$.
\end{theorem}
\subsection{The homotopy transfer theorem}
Now, we recall how to transfer $L_\infty$-structure from homotopy retract.
Let $(L, \ell_k)$ be an $L_\infty$-algebra. Consider the following diagram
\begin{equation}\label{hr}
\xymatrix{ \ar@(ul,dl)@<-5.6ex>[]_K  & (L,d )
\ar@<0.75ex>[r]^-q & (H,0) \ar@<0.75ex>[l]^-i }
\end{equation}
in which $H=H(L,\ell_1)$, $i$ is a quasi-isomorphism, $qi={\rm id}_H$ and $K$ is a chain homotopy between ${\rm id}_L$ and $iq$, i.e.,  ${\rm id}_L-iq=\partial K+K\partial $. We encode this data as $(L,i,q,K)$ and call it a {\em homotopy retract of $L$}. In this setting, the classical {\em Homotopy Transfer Theorem}  reads :
\begin{theorem}\cite[Theorem 10.3.5]{12L-V}\label{ht}

There exists an $L_{\infty}$-algebra structure $\{\ell'_k\}$ on $H$, unique up to isomorphism,
    and  $L_{\infty}$ quasi-isomorphisms
$$
\xymatrix{ (L,\partial )
\ar@<0.75ex>[r]^-Q & (H,\{\ell'_k\}) \ar@<0.75ex>[l]^-I }
$$
such that $I^{(1)}=i$ and $Q^{(1)}=q$. Moreover, the
transferred higher brackets can be explicitly described by the following
\begin{equation}\label{f1}
\ell'_k=\sum_{T\in \mathscr{T}_k}\frac{\ell_T}{|\Aut (T)|},
\end{equation}
where $\mathscr{T}_k$ is the set of isomorphism classes of directed binary rooted trees with k leaves, $\Aut (T)$ is the automorphism group of the tree $T$, and $\ell_T$ is described below. $\hfill\qed$
\end{theorem}
We describe here the item $\ell_T$ in the formula (\ref{f1}). For $T\in \mathscr{T}_k$, we define  the linear map $\ell_T$ as follows. The leaves of the tree $T$ are labeled by the map $i$, each internal edge is labeled by the chain homotopy $K$ and the root edge is labeled by $q$. By moving down from the leaves to the root, we define the linear map $\ell_T$.  For example, the tree 
$$
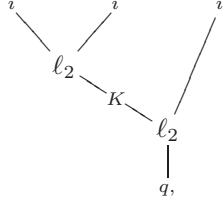
\xymatrixcolsep{1pc}
\xymatrixrowsep{1pc}
\entrymodifiers={=<1pc>} \xymatrix{
*{^i}\ar@{-}[dr] & *{} & *{^i}\ar@{-}[dl] & *{} & *{^i} \ar@{-}[ddl]\\
*{} & {\ell_2} \ar@{-}[drr]|K & *{} & *{} & *{} \\
*{} & *{} & *{} & \ell_2\ar@{-}[d] & *{} \\
*{} & *{} & *{} & *{_{\stackrel{}{q,}}} & *{} \\
}
$$
giving the map
$$\ell_T=q\circ \ell_2\circ (K\otimes id)\circ(\ell_2\otimes id)\circ (i\otimes i\otimes i).$$

\section{Proof of Theorem \ref{theorem 1} and applications}
We begin with the following
\begin{lemma}\label{lemma3}
Let $(\Lambda V, d)$ be the commutative differential graded algebra of finite type of the form (\ref{eq1}), i.e.,  
$$(\Lambda V, d)=(\Lambda(P\oplus Q), dP=0, dQ\subset \Lambda P).$$
The $\emph{Chevalley-Eilenberg construction}$ gives an $L_\infty$-algebra $({L,\{\ell_k\}_{k\ge1}})$ such that $\mathscr{C}^*(L)=(\Lambda V,d)$.
Then the $L_\infty$-structure $\{\ell_k\}_{k\ge1}$ satisfies
$$[\cdots,[\cdots],\cdots]=0.$$
\end{lemma}
 \begin{proof} 
 Recall that $V$ and $sL$ (the suspension of $L$) are dual graded vector spaces and the $L_\infty$-structure $\{\ell_k\}_{k\ge1}$ is determined by the differential $d=d_1+d_2+\cdots$, where derivations $d_k$ raise the wordlength by $k-1$. More precisely, we have 
 \begin{align*}\label{331}
\langle d_k v;sx_1\wedge ... \wedge sx_k \rangle
&=\pm \langle v;s\ell_k( x_1,\dots, x_k) \rangle\\
&=\sum_{\sigma\in S_k}\pm \langle v_{\sigma(1)};sx_1\rangle \cdots \langle v_{\sigma(k)};sx_k\rangle,
\end{align*}
where  $d_kv= v_1\cdots v_k$, $\langle\cdot~;~\cdot\rangle$ is the Sullivan pairing, $v$, $v_1$, $\cdots$, $v_k\in V$ and $x_1$, $\cdots$, $x_k\in L\cong s^{-1}V^\#$. 
 
A straightforward computation shows that
$$ [\cdots,q,\cdots]=0,\quad q\in s^{-1}Q^\#,$$
and
  $$[p_1,\cdots,p_i]\in \Q q_1\oplus\cdots\oplus \Q q_j,\quad p_1, \cdots,p_i,\in s^{-1}P^\#, q_1,\cdots,q_j\in s^{-1}Q^\#.$$
It follows that
$$[x_1,\cdots,[x_2,\cdots],\cdots]=0, \quad x_1,x_2\in s^{-1}V^\#.$$
 \end{proof}

Suppose that $L$ is the minimal $L_\infty$-model of $Y$ and $(A, d)$  is a finite dimensional
commutative differential graded algebra model of $X$. By Theorem \ref{thm2}, the $L_\infty$-algebra $(A\otimes L,\{\ell_k\}_{k\ge1})$ is the $L_\infty$-model of $\map(X, Y)$. 

Let $H=H(A, d)$, then we have the following decomposition 
\begin{equation}\label{eq3}
A=B\bigoplus dB\bigoplus H
\end{equation}
with basis $\{b_i\}$, $\{\theta b_i\}$, $\{h_j\}$. Note that $d=0$ in $H$ and $d\colon B\to dB$ is an isomorphism. It is easy to check that the decomposition (\ref{eq3}) induces a homotopy retract:

\begin{equation*}
\xymatrix{ \ar@(ul,dl)@<-5.6ex>[]_{K'}  & (A\otimes L,\ell_1 )
\ar@<0.75ex>[r]^-{q'} & (H\otimes L,\ell'_1), \ar@<0.75ex>[l]^-{i'} }
\end{equation*}
where $K'=K\otimes id$, $q'=q\otimes id$, $i'=i\otimes id$.

By Theorem \ref{ht}, we obtain an $L_\infty$-structure $ \{\ell'_k\}_{k\ge1}$ on $H\otimes L$ 
and the following quasi-isomorphism between $L_\infty$-algebras
$$(A\otimes L,\{\ell_k\}_{k\ge1})\stackrel{\simeq}{\longrightarrow}(H\otimes L,\{\ell'_k\}_{k\ge1}).$$
Then, we have

\begin{lemma}\cite{16B-G}
The $L_\infty$-algebra $H\otimes L$ is an $L_\infty$-model of $\map(X, Y)$.
\end{lemma}
Now, we prove our main result.
\begin{proof}[proof of the Theorem \ref{theorem 1}]

To prove this theorem, it is now enough to prove that
$$\ell'_k(h_1\otimes x_1,\cdots,h_k\otimes x_k)=h_1\cdot_H h_2\cdot_H\ldots\cdot_H h_k\otimes [x_1,\ldots, x_k],\quad k\ge2$$
where $h_1\otimes x_1$, $\cdots$, $h_k\otimes x_k\in H\otimes L,$ and $\cdot_H$ denotes the product of $H$.

We proceed by induction on k. The case $k =1$ follows immediately from the
definition. For $k=2$, the explicit formula for $\ell'_2$ is provided by the tree

$$
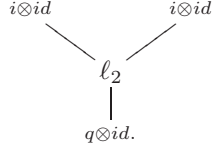
\xymatrixcolsep{1pc}
\xymatrixrowsep{1pc}
\entrymodifiers={=<1pc>} \xymatrix{
*{^{i\otimes id}}\ar@{-}[dr] & *{} & *{^{i\otimes id}}\ar@{-}[dl]\\
*{} & \ell_2\ar@{-}[d] & *{} \\
*{} & *{_{\stackrel{}{q\otimes id.}}} & *{} \\
}
$$
Note that $q\circ i=id_H$, we have 
\begin{align*}
\ell'_2(h_1\otimes x_1,h_2\otimes x_2)&=
(q\otimes id)\circ\ell_2\circ (i\otimes id, i\otimes id)(h_1\otimes x_1, h_2\otimes x_2)\\
&=(q\otimes id)\circ\ell_2(h_1\otimes x_1, h_2\otimes x_2)\\
&=q(h_1\cdot_Ah_2)\otimes [x_1,x_2])\\
&=h_1\cdot_Hh_2\otimes [x_1,x_2],
\end{align*}
where $h_1$, $h_2\in H$, $x_1$, $x_2\in L$, and $\cdot_A$ denotes the product of $A$.

For $k=3$, the explicit formula for $\ell'_3$ is provided by the following trees

$$
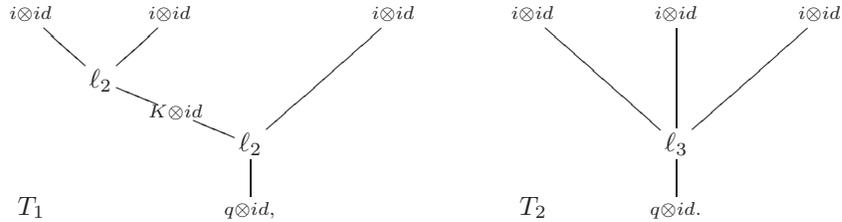
\xymatrixcolsep{1pc} \xymatrixrowsep{1pc} \entrymodifiers={=<1pc>}
\xymatrix{
 &*{^{i\otimes id}}\ar@{-}[dr]& & *{^{i\otimes id}}\ar@{-}[dl] & *{}&*{}& *{^{i\otimes id}}\ar@{-}[ddll]&
 & *{^{i\otimes id}}\ar@{-}[ddrr]&& *{^{i \otimes id}}\ar@{-}[dd] & *{}&*{^{i \otimes id}}\ar@{-}[ddll]  \\
 && \ell_2\ar@{-}[drr]|{K \otimes id}&   && && & && *{} && \\
 &&*{} & *{} & \ell_2\ar@{-}[d] && & &&*{} & \ell_3\ar@{-}[d] & *{} & \\
 &T_1&*{} & *{} & *{_{\stackrel{}{q \otimes id,}}}& & & &T_2  &*{} & *{_{\stackrel{}{q\otimes id.}}} & *{} &
}
$$
By Lemma \ref{lemma3}, we have
$$[[x_1, x_2],x_3]=0,$$
where $x_1$, $x_2$, $x_3\in L$.
Note that the tree $T_1$ has more than one vertex, we have
\begin{align*}
\ell'_{T_1}(h_1\otimes x_1,h_2\otimes x_2, h_3\otimes x_3)&=(q\otimes id)\circ\ell_2\circ(K\otimes id, id)\circ(\ell_2, id)\circ\\ &(i\otimes id, i\otimes id, i\otimes id)(h_1\otimes x_1, h_2\otimes x_2, h_3\otimes x_3)\\
&=(q\otimes id)\circ\ell_2\circ(K\otimes id, id)(\pm h_1h_2\otimes [x_1, x_2], h_3\otimes x_3)\\
&=(q\otimes id)\circ(\pm K(h_1h_2)h_3\otimes[[x_1, x_2],x_3]\\
&=0,
\end{align*}
where $h_1$, $h_2$, $h_3\in H$, $x_1$, $x_2$, $x_3\in L$. 
This implies that the explicit formula for $\ell'_3$ only depends on the tree $T_2$. 

For $k\ge4$, the formula for $\ell'_k$ is determined by the set of rooted trees with k leaves:
\begin{displaymath}
{\mathscr{T}_k=\bigg\{ \arbreABC}\ ,\  {{}\cdot}\  {{}\cdot }\  {{}\cdot }\ ,{{}\arbreCBA }\bigg\}\ .
\end{displaymath}
Recall that the $k$-corolla of $\mathscr{T}_k$ is the last element in the
above set. Note that the $k$-corolla is the only one which has one vertex. Consider $T_i\in \mathscr{T}_k$ which has at least two vertices. Then by a similar proof as the case $k=3$, we have
$$\ell'_{T_i}=0.$$
It follows that $\ell'_k$ only depends on the k-corolla and $$\ell'_k(h_1\otimes x_1,\cdots,h_k\otimes x_k)=h_1\cdot_H h_2\cdot_H\ldots\cdot_H h_k\otimes [x_1,\ldots, x_k],\quad k\ge3.$$ Thus the theorem is proved.
\end{proof}

\begin{example}
Consider the rational space $X$ with the same  cohomology algebra as $S^2\vee S^2\vee S^5$ (in fact, there are two rational homotopy types whose cohomology algebra is $H^\ast(S^2\vee S^2\vee S^5)$. Let $Y$ be the rational space with Sullivan model $(\Lambda(x, y, z),d)$ with $|x|=3$, $|y|=5$, $|z|=7$ and $dx=dy=0$, $dz=xy$. By Theorem \ref{theorem 1}, $\map(X, Y)\simeq\map(S^2\vee S^2\vee S^5, Y)$. Let $H=H^\ast(S^2\vee S^2\vee S^5)=\Q1\oplus\Q e_2\oplus\Q e'_2\oplus \Q e_5$, and let $L=\Q x_2\oplus\Q y_4\oplus \Q z_6$ with $[x_2, y_4]=z_6$ be the $L_\infty$-model of $Y$, where the subscripts denote degrees. Then $H\otimes L$ is the $L_\infty$-model of $\map(X, Y)$.  A basis for $H\otimes L$ is given by
\begin{align*}
\text{degree 0: } & e_2\otimes x,~e'_2\otimes x; \\
\text{degree 1: } &  e_5\otimes z;\\
\text{degree 2: } &1\otimes x,~e_2\otimes y,~e'_2\otimes y;\\
\text{degree 4: } & 1\otimes y,~e_2\otimes z,~e'_2\otimes z;\\
\text{degree 6: } &1\otimes z.
\end{align*}
The non-trivial $L_\infty$-structure is described by
$$[e_2\otimes x, 1\otimes y]=e_2\otimes z,~[e'_2\otimes x, 1\otimes y]=e'_2\otimes z,~[ 1\otimes x,e_2\otimes y]=e_2\otimes z,$$
$$~[ 1\otimes x,e'_2\otimes y]=e'_2\otimes z,~~[ 1\otimes x,1\otimes y]=1\otimes z.$$
A direct computation shows that $\map(X, Y)$ has two components. One has the  same rational homotopy type with $$K(\Q, 2)\times Z,$$where $Z$ 
is the rational space with minimal model
$$(\Lambda (a_1,a'_1,a_3,a'_3,a''_3,a_5,a'_5,a''_5,a_7),d)$$
with $da'_5=a_1a_5+a_3a'_3,$ $da''_5=a'_1a_5+a_3a''_3$ and  $da_7=a_3a_5$.
Another has the  same rational homotopy type with $$S^3\times S^3\times S^7\times Z',$$
where $Z'$ 
is the rational space with minimal model $$(\Lambda (a_1,a'_1,a_5,a'_5,a''_5,a_7),da'_5=a_1a_5,da''_5=a'_1a_5).$$

\end{example}
\subsection*{H-space structures of mapping space}

An interesting question is to determine whether a mapping space is of the rational homotopy type of an H-space. We show 
\begin{corollary}Let $Y$ be as in the Theorem \ref{theorem 1}. If $H^\ast(X)\simeq H^\ast(Y)$, then there exist a component of mapping space $\map(X,Y)$ which is a rational H-space.
\end{corollary}

\begin{proof}
By Therorem \ref{theorem 1}, $\map(X,Y)\simeq\map(Y,Y).$ Then, there exist a component $\map(X,Y;f)$ of $\map(X,Y)$ such that $\map(X,Y;f)\simeq\map(Y,Y;id)$. Note that $\map(Y,Y;1)$ is a group-like space \cite[Theorem 3.6]{08L-S}. It follows that $\map(X,Y;f)$ is a rational H-space.
\end{proof}
\begin{example}
Let $Y$ be the rational space with Sullivan model $(\Lambda(x, y, z),d)$ with $|x|=3$, $|y|=5$, $|z|=7$ and $dx=dy=0$, $dz=xy$. Let $X$ with the commutative differential graded algebra model $H$ which is isomorphic to $H^\ast(Y)$. Note that $H=Q1\oplus \Q\overline{x}\oplus\Q\overline{y}\oplus\Q\overline{xz}\oplus\Q\overline{yz}\oplus\Q\overline{xyz}$ with the only non-trivial products $\overline{x}\cdot\overline{yz}=\overline{xz}\cdot\cdot\overline{y}=\overline{xyz}$. The spaces $X$ and $Y$ are not rational homotopy equivalent since $Y$ is not formal. A straightforward computation shows that the mapping space has three components and one of them is a group-like space:
$$K(\Q, 2)\times K(\Q, 2)\times K(\Q, 3)\times K(\Q, 7).$$
\end{example}

\subsection*{Classifying the component of the mapping spaces}
Let $X$ and $Y$ be as in the Theorem \ref{theorem 1}. A fundamental problem is to classify the components of $\map(X, Y)$ up to homotopy type. By Theorem \ref{theorem 1}, we have an $\L_\infty$-model $H\otimes L$ for the $\map(X, Y)$. Recall (\cite[Theorem 1.5]{15AB}) that $$[X, Y]\cong\mathscr{MC}(H\otimes L),$$ where the moduli space $\mathscr{MC}(H\otimes L)=\mathsf{MC}(H\otimes L)/\sim$ the quotient set of equivalence classes of  \emph{Maurer-Cartan element}. Let $z\in \mathsf{MC}(H\otimes L)$, we have a new $\L_\infty$-algebra $H\otimes L^z$ \cite[Proposition 4.4]{09EG}. The truncated and twisted $L_\infty$-algebra $(H\otimes L^z)_{\ge0}$ is an $L_\infty$-model for the component $\map(X, Y; f)$, where the  \emph{Maurer-Cartan element} $z$ represents the map $f\colon X\rightarrow Y$.

Given an isomorphism $\psi\colon H^\ast(X)\rightarrow H^{\ast} (X)$ between the cohomology algebra. Consider the linear map $\psi\otimes id\colon H\otimes L\rightarrow H\otimes L$, we have
\begin{proposition} If $h\otimes l\in \mathsf{MC}(H\otimes L)$, then $\psi\otimes id(h\otimes l)=\psi(h)\otimes l\in \mathsf{MC}(H\otimes L)$.
\end{proposition}
\begin{proof}
By Theorem \ref{theorem 1}, the $L_\infty$-structure $\{l_k\}_{k\ge1}$ of $H\otimes L$ is depended on the product of $H$ and the $L_\infty$-structure of $L$, i.e. $l_k(h\otimes l,\ldots, h\otimes l)=h\cdot_H h\ldots h\otimes [l,\ldots, l]_L$. Thus, 
\begin{align*}0&= \psi\otimes id\big(\sum_{k \geq 0} \frac{1}{k!} l_k(h\otimes l,\cdots,h\otimes l)\big)\\&= \sum_{k \geq 0} \frac{1}{k!} l_k(\psi(h)\otimes l,\ldots,\psi(h)\otimes l)\\&=\sum_{k \geq 0} \frac{1}{k!} \psi(h\cdot_H h\ldots h)\otimes [l,\ldots, l].\end{align*}
It follows that $\psi(h)\otimes l\in \mathsf{MC}(H\otimes L)$.
\end{proof}

Let the  \emph{Maurer-Cartan elements} $h\otimes l$ and $\psi(h)\otimes l$ repesent the maps $f$ and $f^\prime$ respectively. Note that, as graded vector spaces,
$$\psi\otimes id\colon(H\otimes L^{h\otimes l})_{\ge0}\stackrel{\cong}{\longrightarrow}(H\otimes L^{\psi(h)\otimes l})_{\ge0}.$$
It is now easily checked that this isomorphism is compatible with the  $L_\infty$-structure, i.e. $(H\otimes L^{h\otimes l})_{\ge0}$ and $(H\otimes L^{\psi(h)\otimes l})_{\ge0}$ are isomorphism of $L_\infty$-algebra. Thus, 
$$\map(X, Y; f)\simeq_\Q\map(X, Y; f^\prime).$$

\subsection*{Acknowledgments}
We thank the anonymous referee for helpful comments and suggestions.

\end{document}